\documentclass[a4paper,10pt]{article}      
\usepackage{graphicx}

\usepackage[utf8x]{inputenc}
\usepackage{amssymb,amsmath,amsfonts}
\usepackage{theorem}

\usepackage{tikz}
\usepackage{graphicx}
\usepackage{pgf,tikz}
\usetikzlibrary{arrows}
\newcommand{\Z}{\mathbb{Z}}
\newcommand{\F}{\mathbb{F}}

\newtheorem{theorem}{Theorem.}[section]

\newtheorem{lemma}[theorem]{Lemma.}
\newtheorem{definition}[theorem]{Definición.}
\newtheorem{proof}[theorem]{Demostración.}
\newtheorem{remark}[theorem]{Nota.}

\begin{document}

\title{A module structure over maximal consistent sets}

\author{Kevin D\'avila Castellar \and Ismael Guti\'errez Garc\'ia}

\date{Received: date / Accepted: date}

\maketitle

\begin{abstract}
It is shown the construction of a module structure \cite{Roman} with universe over a set of a particular kind of mathematical proofs, the base ring of this module will be built on a maximal consistent extension of a set of propositions, this provides the possibility to do some linear algebra on proofs. It will also be presented an algorithmic proceeding in order to deal with these particular type of deductions.

\end{abstract}

\section{Introduction}
As it is used in propositional calculus, we need to accept the introduction of an elementary set called alphabet, which elements are usually known as letters, and based on this set and some symbols called binary and unary connectives we define the set of formulas as well defined strings of letters and connectives. With the settlement of semantic rules for
this calculus arises the concept of completeness of a set of logical connectives, concerning whether it is possible or not to rewrite any well constructed formula in terms of these connectors.\\
With this last problem in mind we will make the set $\mathcal{A}$ our infinite and enumerable alphabet, and will consider the symbols $\vee$ , $\wedge $ and $\neg $ in the usual way they are defined.
It should be noted that these are in fact a complete set of connectors.

\section{Preliminaries} 

\begin{definition}
 We will start by defining the $n$-levels of propositions as follows:
\begin{align*}
\F_0 & = \mathcal{A}\\
\F_1 & = \{\neg \alpha \mid \alpha \in \F_0 \} \cup \{(\alpha\ast\beta) \mid \alpha,\beta \in \F_0 \}\cup \F_0\\
     & \ \vdots\\
\F_n & = \{\neg \alpha \mid \alpha \in \F_{n-1}\}\cup \{(\alpha\ast\beta) \mid \alpha,\beta \in \F_{n-1}\} \cup \F_{n-1},
\end{align*} 
where $\ast$ runs through our set of connectives \cite{Hinman}.
\end{definition}

As suggested by the previous definition our set of formulas is $\bigcup_{n \in \mathbb{N}_0} \F_n$ and will be denoted by $\F_{(\infty)}$. We will now describe the fundamental concepts of this presentation.

\begin{definition}
The set that will define the propositions of our main system is based on the previous construction of $\F_{(\infty)}$ and it will be given by
\[\F'_{(\infty)} =\  \F_{(\infty)}/\equiv \ = \{\ [\alpha]_{\equiv } \mid \alpha \in \F_{(\infty)}\}.\]
where $\alpha \equiv \beta$ if and only if $\hat{v}(\alpha) = \hat{v}(\beta)$ for all boolean valuation $v: \F_0 \longrightarrow \left\{0,1\right\}$.
\end{definition}

\begin{definition}
The finite conjunction and disjunction of elements from $\F'_{(\infty)}$ will be understood in the next way:
\[\bigwedge_{h\in H} \gamma_h := \left[ \bigwedge_{h \in H} \alpha_h \right]_{\equiv}\]
and
\[\bigvee_{h \in H} \gamma_h := \left[ \bigvee_{h \in H} \alpha_h \right]_{\equiv},\]
where $\gamma_h =[\alpha_h]_{\equiv}$.
Negation is defined analogously. $\F'_{(\infty)}$ with this disjunction, conjunction and negation is usually called the Lindenbaum algebra over $\mathcal{A}$ \cite{Hinman}.
\end{definition}

\begin{definition}
Let $\Sigma \subseteq \F'_{(\infty)} $. A succession in the form $\langle \gamma_i\rangle^n$, with $\gamma_i \in \F'_{(\infty)}$, is called a deduction of $\sigma$ from $\Sigma$ if and only if
\[\gamma_n =\sigma \]
and for each $i \in \{1,...,n\}$, one of the following propositions holds:
\begin{enumerate}
\item[$(a)$] $\gamma_i \in \Sigma$
\item[$(b)$] exists $\gamma \in \Sigma$ such that $\gamma \vdash \gamma_i$
\item[$(c)$] exists $H \subseteq \{\ 1, \ldots ,i-1 \}$ such that $\bigwedge_{h \in H} \gamma_h = \gamma_i$ or $\bigwedge_{h \in H} \gamma_h \vdash \gamma_i$
\item[$(d)$] exists $H \subseteq \{\ 1, \ldots ,i-1 \}$ such that $\bigvee_{h \in H} \gamma_h = \gamma_i$ or $\bigvee_{h \in H} \gamma_h \vdash \gamma_i$,
\end{enumerate}

where the forementioned inferences are referred to the use of the next rules:

\begin{align*}
 \alpha \wedge \beta & \vdash \alpha \\
 \alpha & \vdash \alpha \vee \beta.
\end{align*}
\end{definition}

\begin{theorem}
The existence of any clasical nontrivial deduction can be proved by the construction of a deduction on $\F'_{(\infty)}$ (We call trivial deduction to any proof consisting of just one proposition).
\end{theorem}

\begin{proof}
It is known that the axiomatic standard system of propositional calculus is equivalent to the set of the next rules of inference:
\begin{align*}
        \text{Modus} & \ \text{Ponens}\\
 \alpha \wedge \beta & \vdash \alpha \\
 \alpha \wedge \beta & \vdash \beta\\ 
 			  \alpha & \vdash \alpha \vee \beta\\ 
 			   \beta & \vdash \alpha \vee \beta\\
   \{\alpha, \beta\} & \vdash \beta \vee \alpha\\
  \{\alpha \vee \beta, \neg \alpha\} & \vdash \beta\\
\{\alpha \Rightarrow \beta, \neg \beta\} & \vdash \neg \alpha\\
\alpha & \vdash \neg \neg \alpha\\
\neg \neg \alpha & \vdash \alpha.
\end{align*}
See for example \cite{Caicedo}, for which there are proofs of their respective equivalence classes, this fact can be easily verified. Therefore any process that makes use of the axioms or the MP rule 
has a proof of his equivalence classes in our system and this ends the proof. 
\end{proof}

With this theorem, we are now aware of the sufficiency of our theory.

\section{Construction of the structure}

\begin{definition}
 A map $\phi : \{1,\ldots,n\}  \longrightarrow \mathcal{P} (\{1,\ldots,n\} ) \cup \{0\}$ will be called interpretation or reading of a deduction $\langle \gamma_i\rangle^n$ of $\gamma$ from $\Sigma$ if it satisfies:
\begin{enumerate}
\item[$(a)$] $\phi (i)=0$ if $\gamma_i \in \Sigma$ or exists $\gamma \in \Sigma$ such that $\gamma\vdash \gamma_i$.
\item[$(b)$] $\phi(i)=H\subseteq \{\ 1,\ldots,i-1 \}$ if any of the following conditions is satisfied
	\begin{itemize}
		\item $\bigwedge_{h \in H} \gamma_h = \gamma_i$ or $\bigwedge_{h \in H} \gamma_h \vdash \gamma_i$.
		\item $\bigvee_{h \in H} \gamma_h = \gamma_i \vee \bigvee_{h \in H} \gamma_h \vdash \gamma_i$.
	\end{itemize}
\end{enumerate}
\end{definition}

\begin{definition}
 A series in the form $\langle \gamma_i, \phi(i)\rangle^n$ will be called interpreted deduction if $\langle \gamma_i\rangle^n$ is a deduction of $\gamma_n$ and $\phi$ is an interpretation of this deduction.
\end{definition}

The concept of interpretation is not an artificial one, in fact we will propose an induced interpretation that can be built on any arbitrary deduction.

\begin{definition}
If $\langle \delta_i \rangle^n$ is a deduction, it is known that it has to be in the form  $\langle \delta_i \rangle^n = \langle\langle \delta_i \rangle^{n-1}, \sigma \rangle$, where the series of size $n-1$ is a deduction by itself. We will define then
\[ \Omega^{\sigma}:= \left\{ \left\{ \delta_h \mid h \in H \subseteq \left\{1,\ldots,n-1\right\} \right\} \mid \bigwedge_{h \in H} \delta_h = \sigma \text{ or } \bigwedge_{h \in H} \delta_h \vdash \sigma \text{ or } \bigvee_{h \in H} \delta_h = \sigma \text{ or } \bigvee_{h \in H} \delta_h \vdash \sigma \right\}.\]

$\Omega^\sigma$ is the set constituted by the sets of proposition that result into $\sigma$ by conjunction or disjunction.
\end{definition}

\begin{definition}
 Let us define 
\[\mu : \langle \delta_i \rangle^n\longrightarrow \mathbb{Z}^+ \]
\[\delta_j \longmapsto p_j,\]
where $p_j$ is the $j$-th prime number and $\langle \delta_i \rangle^n$ is a deduction. We define then the next function
\[\Gamma_{\sigma} : \Omega^{\sigma} \longrightarrow \Z^+\]
\[\{\delta_h \mid h \in H \}\ \longmapsto \prod_{h \in H} \mu(\delta_h).\]
This function $\Gamma_\sigma$ assigns the product of their corresponding primes to the sets of propositions that result in $\sigma$ by conjunction or disjunction.
\end{definition}

Finally our interpretation will be given as follows:

\begin{definition}
 If $\langle \delta_i \rangle^n$ is a deduction of $\sigma$ and  $\Omega^\sigma$ is as in the previous definition then:
\[\phi_{in}^{\langle \delta_i \rangle^n} (n) := 0 \ \ \text{if} \ \ \Omega^\sigma=\emptyset.\]
\[\phi_{in}^{\langle \delta_i \rangle^n} (n) := H \ \ \text{if} \ \ \Omega^\sigma \neq \emptyset,\]
where $H$ is such that:
\[\Gamma^\sigma ( \{\delta_h \mid h \in H\})= \textup{max}(\Gamma^\sigma (\Omega^\sigma))\]
i.e. $\phi_{in}$ identifies the element of $\Omega^\sigma$ with the greatest product of associated primes.
\end{definition}

The recursive application of the procedure on the elements indexed by $H$ which is a process that has to be finite because of the finiteness of the deduction, brings as 
a result after the assignment of zero to the indices of the remaining premises the complete definition of our interpretation.\\
However the current idea of deductions allows several ways of redundancy, in the sense that there might be deductions based on the same reasoning but different because of their forms instead of their content.

We will define next a new set of proofs trying to avoid this problem. We emphasize that on the past we made use of the term deduction, this terminology will help us to avoid any confusions.

\begin{definition}
Let $\langle \gamma_i, \phi(i)\rangle^n$ be an interpreted deduction of $\gamma_n$ from $\Sigma$ a subset of $\F'_{(\infty)}$. We define:
\[\quad R(\gamma_u) := 
  \left\{
   \begin{aligned}
    \left\{\emptyset\right\} \text{ if } \phi(u)=0 \\
     \left\{\left\{\gamma_h , R(\gamma_h)\right\} \mid h \in H\right\} \text{ if } \phi(u)=H.
   \end{aligned}
   \right.\]
\[R_u^{\langle \gamma_i, \phi(i)\rangle^n} := \{ \gamma_u, R(\gamma_u)\}.\]
Finally $R$ will be a proof of $\gamma_n$ from $\Sigma$ if $R=R_{n}^{\langle \gamma_i, \phi(i)\rangle^n}$ for some interpreted deduction $\langle \gamma_i, \phi(i)\rangle^n$.
\end{definition}

It is important to realize the dependence between the proof and the interpretation of the deduction that originated it.

This concept of proof creates a relation over the set of deductions as follows:

\begin{definition}
 We call two deductions essentially equal if they are related by $I$ where:
\[\langle \delta_i \rangle^n I \langle \alpha_j\rangle^m : \Leftrightarrow R_n^{\langle \delta_i \rangle^n} = R_m^{\langle \alpha_j\rangle^m}.\]
We have omitted the interpretation here for convenience.
\end{definition}

\begin{remark}
 It is clear that the relation $I$ is an equivalence relation. Also the map that assigns the corresponding proof to an specific deduction is a complete invariant for the relation $I$.
\end{remark}

Hereafter we will use the following notation:
\[\mathcal{M}^{\Sigma} := \left\{ R_{n}^{\langle \beta_i,\phi(i)\rangle^n} \mid \langle \beta_i,\phi(i)\rangle^n \in \mathcal{K}^{\Sigma}\right\},\]
where $\mathcal{K}^\Sigma$ is the set of all deductions from $\Sigma$.

If we now make:
\[\rho := \left\{[\langle\beta_i,\phi(i)\rangle^n]_I \mid \langle\beta_i,\phi(i)\rangle^n \in \mathcal{K}^\Sigma \right\},\]
there is a clear natural bijection between $\mathcal{M}^\Sigma$ and $\rho$ that can be used to move the structure we are defining over $\mathcal{M}^\Sigma$ to $\rho$.\\\\
At this point we have already defined $\phi_{in}$, however for it to be an actual interpretation we need to guarantee the validity of the implication $\phi_{in}(i)=0 \Rightarrow \gamma_i \in \Sigma$. This is the reason that moves us to consider a maximal consistent extension $\Sigma'$ of $\Sigma$ that is known to be closed under deduction and this clearly makes the previous implication true. 
\begin{theorem}\label{3.10}
If we define $\tilde{\vee}$ and $\tilde{l}$ by:
\begin{align*}
 \tilde{\vee} : \Sigma'\times\Sigma'& \longrightarrow \Sigma'\\
 (\alpha,\beta) &\longmapsto \alpha \vee \beta 
\end{align*} 
and
\begin{align*}
 \tilde{l} : \Sigma'\times\Sigma' &\longrightarrow \Sigma'\\
 (\alpha,\beta) &\longmapsto \alpha l \beta :=(\neg \alpha \vee \beta)\wedge (\neg \beta \vee \alpha)
\end{align*}
then the structure $\left(\Sigma', \tilde{l}, \tilde{\vee} \right)$ becomes a commutative ring.
\end{theorem}

\begin{proof}
$\left(\Sigma',\tilde{l}\right)$ is an abelian group. Indeed $l$ is clearly a well defined binary operation because of the maximal consistency of $\Sigma'$. Besides  if $\alpha, \beta \in \F_{(\infty)}$ we have:
\[\alpha l' (\beta \vee \neg \beta) \equiv \alpha,\]
where $l'$ is the binary connective corresponding to $l$ in $\F_{(\infty)}$. We can deduce from this that:
\[ [\alpha]_{\equiv} l [\beta \vee \neg \beta]_{\equiv} = [\alpha]_{\equiv}\]
therefore
\[\gamma l [\beta \vee \neg \beta]_{\equiv} = \gamma \textup{ for all } \gamma \in \F'_{(\infty)} \]
and given that
\[\alpha l' \alpha \equiv \beta \vee \neg \beta\]
we have
\[ [\alpha]_{\equiv} l [\alpha]_{\equiv}=[\beta \vee \neg \beta]_{\equiv}\]
i.e $\gamma l \gamma = [\beta \vee \neg \beta]_{\equiv}$ for all $\gamma \in \F'_{(\infty)}$. Finally associativity and commutativity can be easily verified and we have now the desired result.

Thus following the same reasoning is evident that $\left(\Sigma', \tilde{\vee}\right)$ is a commutative semigroup. 
Besides if $\alpha, \beta,\gamma \in \F_{(\infty)}$ then:
\[\alpha \vee (\beta l' \gamma) \equiv (\alpha \vee \beta) l' (\alpha \vee \gamma)\]
in this way
\[[\alpha]_{\equiv} \vee \left([\beta]_{\equiv} l [\gamma]_{\equiv}\right) = \left([\alpha]_{\equiv} \vee [\beta]_{\equiv} \right) l \left([\alpha]_{\equiv} \vee [\gamma]_{\equiv}\right)\]
and this clearly ends the proof.
\end{proof}

\begin{remark}
It is a well known fact from Ring Theory that every ring can be embedded in a ring with identity, by doing this we can find a multiplicative neutrum for $\tilde{\vee}$ that will be denoted by $e^{\vee}$.\\
\end{remark}

We say a proof is less forced than other when the set of premises of the first is smaller than the one of the second.

With this in mind we define what will be called the sum of the module. It is necessary to clarify first, that on the purpose of the development of our theory the proofs with the form  
\[\left\{[\alpha \vee \neg \alpha]_{\equiv}, \Delta\right\}\] 
with $\alpha \in \Sigma'$ will be equivalent to 
\[\left\{[\alpha \vee \neg \alpha]_{\equiv}, \left\{\emptyset\right\}\right\}.\] 
This is something that could perfectly has been done in the very definition of proof or could have been added to the operations that we are about to define.

\begin{definition}
 Let $\left\{\zeta_1, \Delta_1\right\}$ and $\left\{\zeta_2, \Delta_2\right\}$ be proofs of $\zeta_1 \text{and} \zeta_2$ respectively and $\alpha \in \mathbb{F'}_{(\infty)}$, we will define a function on $\left(\Delta_1,\Delta_2\right)$ as follows:
 \[(\Delta_1*\Delta_2)_{\alpha} :=
       \left\{
        	\begin{aligned}
         	\{\emptyset\}, & \text{ if } \Delta_1 = \Delta_2  \text{ and $\alpha$ is a tautology }\\
          	\Delta_1, & \text{ if } \Delta_1 = \Delta_2 \text{ and $\alpha$ is not a tautology }\\
           	\Delta_j, & \text{ if } \Delta_i = \{\emptyset\} \text{ for } i,j\in \{1, 2\} i \neq j\\
            \Delta, & \text{ in any other case }
            \end{aligned}
       \right.\]
where
\[\Delta =
       \left\{
        \begin{aligned}
         (\Delta_1-\Delta_2) \cup (\Delta_2-\Delta_1), & \text{ if } \bigwedge ((\Delta_1-\Delta_2) \cup 
         (\Delta_2-\Delta_1)\cap \mathbb{F'}_{(\infty)}) \vdash \zeta_1 \wedge \zeta_2 \\
         \{\emptyset\}, & \text{ otherwise }.
        \end{aligned}
       \right.\]
\end{definition}

\begin{definition}
The next function will be known from now on as the summation inside the module:
\[+ : \mathcal{K}_i^{\Sigma'} \times \mathcal{M}^{\Sigma'} \longrightarrow \mathcal{M}^{\Sigma'}\]
\[(R_1,R_2)\longmapsto R,\]
where $R$ will be described next:

As the proofs $R_1$ and $R_2$ have to be in the form:
\[R_1 = \{\zeta_1, \Delta_1\}\]
\[R_2 = \{\zeta_2, \Delta_2\}\]
with $\zeta_1, \zeta_2 \in \F'_{(\infty)}$, we will do:
\[R := \{\zeta_1 l \zeta_2, (\Delta_1*\Delta_2)_{\zeta_1 l \zeta_2}\}.\]
\end{definition}

\begin{lemma}
 The map previously described is precisely a well defined binary operation.
\end{lemma}

\begin{proof}
We have clearly
\[R' := \{\zeta_1 \wedge \zeta_2, (\Delta_1*\Delta_2)_{\zeta_1 l \zeta_2}\} \in \mathcal{K}_i^{\Sigma'}\]
and given that  
\[\zeta_1 \wedge \zeta_2 \vdash (\zeta_1\wedge \zeta_2)\vee (\neg \zeta_1 \wedge \zeta_1) \vee ((\zeta_2 \vee \neg \zeta_1) \wedge \neg \zeta_1)\]
and the fact that this expression is equal to $\zeta_1 l \zeta_2$, we easily get the desired result.
\end{proof}

\begin{theorem}\label{3.13}
 $\left(\mathcal{M}^{\Sigma'},+\right)$ is an abelian group.
\end{theorem}

\begin{proof}
It follows from previous lemma that this structure is a magma. The commutativity is inherited from $\left(\Sigma',\tilde{l}\right)$ and the commutativity of *. Associativity can be easily verified. Besides, it is clear that 
\[\{[\alpha \vee \neg \alpha]_{\equiv}, \{\emptyset\}\}\] 
is a neutral element for the operation and the elements of the structure are involutions, this fact can be checked from the definition, this ends the proof.
\end{proof}

As it can be seen this operation farther of building a proof for $\zeta_1 l \zeta_2$, diminishes, if it is possible, the number of unjustified (non-deduced) premises. Although this operation may seem weirdly defined it appears naturally when we are trying to reuse the justifications from the proofs we are trying to add.

We will use from now on, as usual, the following notation.
\[+ (R_1,R_2) = R_1+R_2.\]
Let us now define the scalar product.

\begin{definition}
 The next map will be known from now on as the scalar product:
\[\cdotp  : \Sigma' \times \mathcal{M}^{\Sigma'} \longrightarrow \mathcal{M}^{\Sigma'}\]
\[(\sigma, R) \longmapsto R'.\]
The proof $R'$ will be built based on $R$ and $\{\sigma, \{\emptyset\}\}$. We already know that $R$ has to be in the form
\[R = \{\zeta, \Delta\} \text{ where } \zeta \in \F'_{(\infty)}\] 
then
\[R' := \{\sigma \vee \zeta, \Delta\}.\]
We shall write, as usual $R'=\sigma \cdotp R$.
\end{definition}

\begin{remark}
The image of the previous function is precisely a subset of $\mathcal{M}^{\Sigma'}$.
\end{remark}

\begin{remark}
One can see on the other hand that the next map is an injection.
\[\psi : \Sigma' \longrightarrow \mathcal{M}^{\Sigma'}\]
\[ \gamma \longmapsto \{\gamma, \{\emptyset\}\}.\]
So our scalar product could perfectly has been considered an operation between proofs.
\end{remark}

\begin{lemma}\label{3.16}
In the structure $\left(\mathcal{M}^{\Sigma'},+,\cdotp\right)$ next properties are satisfied:
\begin{align}
(\theta \vee \beta)\cdotp R & = \theta \cdotp (\beta\cdotp R) \text{  for  } \theta, \beta \in \Sigma' \text{ and } R \in \mathcal{M}^{\Sigma'}\\
e^{\vee}\cdotp R & = R \text{  for } R \in \mathcal{M}^{\Sigma'}\\
\alpha \cdotp (R_1 + R_2) & = \alpha \cdotp R_1 + \alpha \cdotp R_2 \text{  for } \alpha \in \Sigma' \text{ and } R_1, R_2 \in \mathcal{M}^{\Sigma'}\\
(\alpha l \beta)\cdotp R & = \alpha \cdotp R + \beta \cdotp R \text{  for } \alpha, \beta \in \Sigma' \text{ and } R \in \mathcal{M}^{\Sigma'}.
\end{align}
\end{lemma}

\begin{proof}
\begin{itemize}
\item[(1)]Follows from scalar product definition.
 
\item[(2)] As $R$ has to be in the form $R=\{\delta, \Delta\}$ with $\delta \in \Sigma'$, then $e^{\vee}\cdotp \{\delta, \Delta\} = \{e^{\vee} \vee \delta, \Delta\} = \{\delta, \Delta\}$, what ends the proof.

\item[(3)] As $R_1$ and $R_2$ has to be in the form $R_1 = \{\delta_1,\Delta_1\}$ and $R_2 = \{\delta_2,\Delta_2\}$ with $\delta_1, \delta_2 \in \Sigma'$, then
\begin{align*}
\alpha\cdotp (R_1+R_2) & = \alpha \cdotp \{\delta_1 l \delta_2, (\Delta_1*\Delta_2)_{\delta_1 l \delta_2}\}\\
                       & = \{(\delta_1 l \delta_2) \vee \alpha, (\Delta_1*\Delta_2)_{\delta_1 l \delta_2}\}\\
                       & = \{(\delta_1\vee \alpha) l (\delta_2\vee \alpha), (\Delta_1*\Delta_2)_{\delta_1 l \delta_2} \}\\
                          & = \alpha \cdotp R_1 + \alpha \cdotp R_2.
\end{align*}
\item[(4)] It is known that 
\[(\alpha l \beta)\cdotp R = \{(\alpha l \beta) \vee \sigma, \Delta\},\] 
where we have previously assumed that $R=\left\{\sigma, \Delta\right\}$ with $\sigma \in \Sigma'$. In the case $\left(\alpha l \beta\right) \vee \sigma$ be a tautology the result is immediate, otherwise one has:
\begin{align*}
\{(\alpha l \beta) \vee \sigma, \Delta\} & = \{(\alpha \vee \sigma) l (\beta \vee \sigma), (\Delta*\Delta)_{(\alpha \vee \sigma) l (\beta \vee \sigma)}\}\\
                                         & = \{\alpha \vee \sigma, \Delta\} + \{\beta \vee \sigma, \Delta\}
\end{align*}
and one has the desired result.
\end{itemize}
\end{proof}

\begin{theorem}
The structure $\left(\mathcal{M}^{\Sigma'},+,\cdotp\right)$ is a $\Sigma'$-module.
\end{theorem}

\begin{proof}
 It follows from theorems \ref{3.10}, \ref{3.13} and lemma \ref{3.16}.
\end{proof}

\begin{theorem}
It is always possible to build a module over the maximal consistent extension of a consistent non empty set of propositions, which elements are proofs on such a extension.
\end{theorem}

\begin{proof}
 It follows from previous results.
\end{proof}
 
\section{ANNEXES}
\subsection{Subproofs processing}
A very desirable characteristic of these proofs, which is inherited from the very definition of deduction, is that it is possible to easily extract the subproofs used in it. An algorithmic proceeding for replacing, eliminate or extract subproofs is presented next.

This example is the one of a replacement of a subproof, the rest can be performed in an analogues way.

On this purpose we will make use of the next notation:
\[\sigma \in_{(\mathcal{L}_j)}^{q} R \Leftrightarrow \exists \mathcal{L}_1,\ldots,\mathcal{L}_{q-1} \mid \sigma \in \mathcal{L}_1 \in \cdots \in \mathcal{L}_{q-1} \in R.\]
Now, let $R_h, R_k$ be elements from $\mathcal{K}_i^{\Sigma'}$ and:
\[\sigma \in_{(\mathcal{L}'_j)}^{q} R_h \wedge \sigma \in_{(\mathcal{L}_i)}^{l} R_k \]
and suppose that:
\[\mathcal{L}_1 \neq \left\{ \sigma, \left\{\emptyset \right\}\right\}.\]
We will make the next construction: 
\[\mathcal{L}''_{2} \in \mathcal{L}''_{3} \in \cdots \in \mathcal{L}''_{q-1},\]
where
\[\mathcal{L}''_{2}=\left( \mathcal{L}'_{2}-\mathcal{L}'_1\right) \cup \left\{\sigma, \bigcup_{i=1}^m \zeta_i \mid \zeta_i \in \mathcal{L}_1^i \wedge \zeta_i \notin \F'_{(\infty)}\right\}\]
and
\[\mathcal{L}''_i:= \left(\mathcal{L}'_i-\mathcal{L}'_{i-1}\right) \cup \mathcal{L}''_{i-1}.\]
The chain $\left( \mathcal{L}''_i\right)$ must replace $\left( \mathcal{L}'_i\right)$ in the proof $R_h$ obtaining $R'_h$ as the desired proof.

\subsection{Proposal for proving the existence of a neutral element for disjunction}
Let
\[\mathcal{V}= \F'_{(\infty)} \cup K_1 \cup K_2,\]
where
\[K_1 := \left\{\bigwedge_{i=1}^\infty \bigvee_{j=1}^\infty \phi_{ij} \mid \phi_{ij} \in \F'_{(\infty)}\right\}\]
and
\[K_2 := \left\{\bigvee_{i=1}^\infty \bigwedge_{j=1}^\infty \phi_{ij} \mid \phi_{ij} \in \F'_{(\infty)}\right\}.\]
We emphasize on the fact that the truth values of $K_1$ y $K_2$ are completely defined since:
\[\quad \hat{v}\left(\bigvee_{j=1}^{\infty} \phi_i\right)= 
  \left\{
   \begin{aligned}
    1 \textup{ si } \exists \phi_h \mid \hat{v}(\phi_h)=1 \\
    0 \textup{ si } \left(\hat{v}(\phi_1),\hat{v}(\phi_2),\ldots\right)=(0,0,\ldots)
   \end{aligned}
   \right.\]
and, consequently
\[\quad \hat{v}\left(\bigwedge_{j=1}^{\infty} \phi_i\right)= 
  \left\{
   \begin{aligned}
    1 \textup{ si } \left(\hat{v}(\phi_1),\hat{v}(\phi_2),\ldots\right)=(1,1,\ldots)\\
    0 \textup{ si } \exists \phi_h \mid \hat{v}(\phi_h)=0
   \end{aligned}
   \right.\]
Note that, if we have $\phi_i=\left[\psi_{j}\right]$ the following proposition is not necessarily true
\[\bigwedge_{i=1}^{\infty} \phi_i=\left[\bigwedge_{i=1}^{\infty} \psi_i\right].\]
We need to clarify that it is also incorrect to expect in general that:
\[\alpha \vdash \beta \wedge \beta \vdash \alpha \Leftrightarrow \alpha \equiv \beta.\]
With the introduction of these kind of formulas it is possible to consider $\alpha_1 \wedge \neg \alpha_1 \wedge \alpha_2 \wedge \neg \alpha_2 \ldots$ as a neutral 
for disjunction, where we are assuming in addition that this formula contains all the elements from $\F'_{(\infty)}$.

\end{document}